\documentclass[12pt,reqno]{amsart}

\usepackage{amsmath,amsfonts,amssymb,amsthm,mathrsfs}
\usepackage{curves}  
\usepackage{epic}  
\usepackage[all]{xy}   
\usepackage{graphicx}
\usepackage{epsfig}

\usepackage[usenames]{color}

\newtheorem{thm}{Theorem}[section]
\newtheorem{prop}[thm]{Proposition}

\newtheorem{lem}[thm]{Lemma}

\newtheorem{rem}[thm]{Remark}
\newtheorem{defn}[thm]{Definition}

\DeclareMathOperator{\diam}{diam}

\newcommand{\Z}{{\mathbb Z}}      \newcommand{\R}{{\mathbb R}}

\def\dist{\qopname\relax o{dist}}

\title[On improved fractional  Sobolev--Poincar\'e inequalities]{On improved fractional  Sobolev--Poincar\'e inequalities}
\author[B{.} Dyda]{Bart{\l}omiej Dyda}
\address[B.D.]{ Institute of Mathematics and Computer Science\\ Wroc{\l}aw University of Technology\\
Wybrze\.ze Wyspia\'nskiego 27,
50-370 Wroc{\l}aw, Poland
}
\email{bdyda@pwr.wroc.pl}

\author[L. Ihnatsyeva]{Lizaveta Ihnatsyeva}
\address[L.I.]{Department of Mathematics and Statistics,
P.O. Box 68, FI-00014 University of Helsinki, Finland} \email{lizaveta.ihnatsyeva@aalto.fi}

\author[A.V. V\"ah\"akangas]{Antti V. V\"ah\"akangas}
\address[A.V.V.]{Department of Mathematics and Statistics,
P.O. Box 68, FI-00014 University of Helsinki, Finland} \email{antti.vahakangas@helsinki.fi}

\date{}

\begin{document}


\begin{abstract}
We prove a certain improved fractional Sobolev--Poincar\'e inequality  on John domains; the proof is based on
the equivalence  of the corresponding weak and strong type inequalities.
We also give necessary conditions for the validity of an improved fractional Sobolev--Poincar\'e inequality,
in particular, we show that a domain having a finite measure and satisfying  this
 inequality, and  a   `separation property',
is a John domain.
\end{abstract}

\maketitle

\section{Introduction}

It is known that the classical Sobolev--Poincar\'e inequality holds on a $c$-John domain $G$ (for the John condition, see Definition \ref{sjohn}).
Namely, if $1\le p<n$, then there exists a constant $C=C(n,p,c)>0$ such that inequality
\begin{equation}\label{sobolev_poincare}
\int_G |u(x)-u_G|^{np/(n-p)}\,dx\le C\bigg(\int_G |\nabla u(x)|^p\,dx\bigg)^{n/(n-p)}
\end{equation}
holds for every $u\in W^{1,p}(G)$.
When $1<p<n$ this result was proved independently by Martio \cite{MR948443} and Reshetnyak \cite{MR851612}.
The method of Reshetnyak is based on the following integral representation in a $c$-John domain:  inequality
\begin{equation}\label{e.repr}
\lvert u(x)-u_G\rvert \le C(n,c)\int_G \frac{\lvert \nabla u(y)\rvert}{\lvert x-y\rvert^{n-1}}\,dy\,,\qquad x\in G\,,
\end{equation}
holds whenever  $u$ is a Lipschitz function on $G$. Bojarski extended inequality \eqref{sobolev_poincare}
to the case $p=1$ by using a certain chaining technique \cite{B}. Later  Haj{\l}asz  \cite{H} showed that inequality \eqref{sobolev_poincare} on John domains for $p=1$
follows from the integral representation \eqref{e.repr}  together with the Maz'ya's truncation argument \cite{M}.
It is also known, that the John condition is  necessary for
the classical Sobolev--Poincar\'e inequality \eqref{sobolev_poincare} to hold, if $G$ is of finite measure
and  satisfies the  separation property; 
this result is due to Buckley and Koskela \cite{MR1359964}.
For instance, every simply connected planar domain satisfies 
 the separation property.

In this paper, we consider certain fractional counterparts of inequality \eqref{sobolev_poincare}.
Let $0<\delta<1$, $1\le p<n/\delta$ and let $G$ be a bounded domain in $\R^n$, $n\ge 2$. It follows from the results obtained by Zhou in \cite[Theorem 1.2]{Z} that
the fractional Sobolev--Poincar\'e inequality
\begin{equation}\label{fractionalqp_nonimproved}
\int_G\vert u(x)-u_G\vert ^{np/(n-\delta p)}\,dx
\le
C
\biggl(\int_G\int_{G}\frac{\vert u(x)-u(y)\vert^p}{\vert x-y\vert ^{n+\delta p}}\,dy\,dx
\biggr)^{n/(n-\delta p)}
\end{equation}
holds for  some $C>0$ and every
$u\in L^p(G)$ if and only if $G$ is Ahlfors $n$-regular.
For example, John domains are $n$-Ahlfors regular, but the converse
fails in general. On the other hand, if we assume that $G$ is a $c$-John domain and  $0<\tau<1$ is given,
then there  exists a constant $C=C(n,\delta,c,\tau,p)>0$ such that a stronger inequality \begin{equation}\label{fractionalqp_I}
\int_G\vert u(x)-u_G\vert ^{np/(n-\delta p)}\,dx
\le
C
\biggl(\int_G\int_{B(x,\tau \dist(x,\partial G))}\frac{\vert u(x)-u(y)\vert ^p}{\vert x-y\vert ^{n+\delta p}}\,dy\,dx
\biggr)^{n/(n-\delta p)}
\end{equation}
holds for every $u\in L^1(G)$.
We  call inequality \eqref{fractionalqp_I} an  improved fractional Sobolev--Poincar\'e inequality, and
it is the main object in this paper.
These inequalities have applications, e.g., in peridynamics, we refer  to \cite{bellido}.
A proof of inequality \eqref{fractionalqp_I} for  $1<p<n/\delta$ is obtained in
 \cite{H-SV} by  establishing a fractional analogue
of the representation formula \eqref{e.repr}  in John domains.

 In  \S\ref{s.fractional},  we extend
the improved inequality \eqref{fractionalqp_I} on John domains to the case $p=1$
by using the mentioned representation formula and the
fractional Maz'ya truncation method from \cite{Dyda3}.
The truncation method is used to show that
inequality \eqref{fractionalqp_I}  is equivalent
to a corresponding weak type inequality, see Theorem \ref{t.truncation}.


We also address
the necessity of John condition for improved fractional Sobolev--Poincar\'e inequalities;
a simple counterexample shows
that the improved inequality
\eqref{fractionalqp_I}
does not hold on  all bounded Ahlfors $n$-regular domains, we refer to \S\ref{s.counter}.
Furthermore,  by adapting the method of Buckley and Koskela  in \S\ref{s.necessary}, we show that the
John condition is necessary for the improved fractional Sobolev--Poincar\'e inequality
\eqref{fractionalqp_I} to hold, if
the domain $G$ has a finite measure and  satisfies the  separation property;
we refer to Theorem \ref{t.necessary}.

When $G$ is a~bounded Lipschitz domain and $\tau\in (0,1]$,  there exists a constant $C>0$
such that, for every $u\in L^1(G)$, the following inequality holds:
\begin{equation}\label{e.lip}
\int_G\int_{G}\frac{\vert u(x)-u(y)\vert^p}{\vert x-y\vert ^{n+\delta p}}\,dy\,dx
\le C
\int_G\int_{B(x,\tau \dist(x,\partial G))}\frac{\vert u(x)-u(y)\vert ^p}{\vert x-y\vert ^{n+\delta p}}\,dy\,dx\,,
\end{equation}
see \cite[formula (13)]{MR2215170}. 
In particular, the fractional Sobolev--Poincar\'e inequalities
\eqref{fractionalqp_nonimproved} and \eqref{fractionalqp_I}
are  equivalent in this case. However, inequality \eqref{e.lip} does not hold for John domains in general; 
we give a counterexample in Proposition~\ref{p.no-comp}.

%
%

\subsection*{Acknowledgements}

L.I. and A.V.V. were supported by the Finnish Academy of Science and
Letters, Vilho, Yrj\"o and Kalle V\"ais\"al\"a Foundation.
B.D. was supported in part by NCN grant 2012/07/B/ST1/03356.

\section{Notation and preliminaries}

Throughout the paper we assume that $G$ is a domain in
$\R^n$, $n\geq 2$.
The distance from $x\in G$ to the boundary of $G$ is  $\dist(x,\partial G)$.
The diameter of a set $A\subset \R^n$ is $\mathrm{diam}(A)$.
The Lebesgue $n$-measure of a  measurable set $A\subset \R^n$ is denoted by $\vert A\vert.$
For a measurable set $A$ with a finite and  nonzero measure we write
$u_A=\lvert A\rvert^{-1}\int_{A}u(x)\,dx$
whenever the integral is defined.
The characteristic function of a set $A$ is written as $\chi_A$.
If  a function $u$ is defined on $G\subset \R^n$ and occurs in a~place where
a~function defined on $\R^n$ is needed, we understand that $u$ is extended by zero to the whole $\R^n$.
We let $C(\ast,\dotsb,\ast)$  denote a constant which depends on the quantities appearing
in the parentheses only.

We use the following definition for John domains;
alternative equivalent definitions may be found in \cite{MR1246886}.

\begin{defn}\label{sjohn}
  A bounded domain $G$ in $\R^n$, $n\ge 2$, is a $c$-John domain  (John domain) with a constant $c\ge 1$, if
  there exist  $x_0\in G$  such that every
  point $x$ in $G$ can be joined to $x_0$ by a rectifiable curve
  $\gamma:[0,\ell]\to G$, parametrized by its arc length, for which
  $\gamma(0)=x$, $\gamma(\ell)=x_0$, and
\[
\dist(\gamma(t),\partial G)\ge t/c\,,
\]
for every $t\in [0,\ell]$.
The point $x_0$ is called a John center of $G$.
\end{defn}

John domains are Ahlfors $n$-regular.

\begin{defn}\label{n-set}
A domain $G$ in $\R^n$ is called Ahlfors $n$-regular,
if there exists a constant $C>0$ such that
 inequality
 $\lvert G \cap B(x,r)\rvert  \geq C r^n$ holds for every $x\in G$ and every $r\in (0,1]$.
 \end{defn}

Let us also recall  the definition of the separation property  from
\cite[Definition 3.2]{MR1359964}.

\begin{defn} A proper domain $G$ in $\R^n$ with a fixed point $x_0\in G$
satisfies a separation property if there exists a constant $C_0>0$ such
that the following holds: for every $x\in G$, there exists a curve $\gamma$ joining $x$ to $x_0$
in $G$ so that for each $t$ either
\[\gamma([0,t])\subset
B:=B(\gamma(t),C_0\dist(\gamma(t),\partial G))\] or each
$y\in\gamma([0,t])\setminus B$ belongs to a different component of
$G\setminus\partial B$ than $x_0$.
\end{defn}

Simply connected
proper planar domains satisfy the separation property. More generally, if $G$ is quasiconformally
equivalent to a uniform domain, then $G$ satisfies
the separation property. For the proofs of these
statements we refer to \cite{MR1359964}.

The Riesz $\delta$-potential $\mathcal{I}_\delta$ with $0<\delta<n$
is defined for an appropriate measurable function $f$ on $\R^n$ and $x\in \mathbb{R}^n$ by
\[
\mathcal{I}_\delta(f)(x) = \int_{\R^n} \frac{f(y)}{|x-y|^{n-\delta}} \,dy\,.
\]
The Riesz $\delta$-potential satisfies the following weak type estimate, see \cite[p. 56]{AH} for the proof.
\begin{thm}\label{AdamsHedberg}
Let $0<\delta<n$. Then there
exists
a constant $C=C(n,\delta)>0$ such that inequality
\[
\sup_{t>0} \big|\{x\in \R^n:\,|\mathcal{I}_\delta(f)(x)|>t\}\big|t^{n/(n-\delta)}\le C\|f\|_1^{n/(n-\delta)}
\]
holds for every $f\in L^1(\R^n)$.
\end{thm}

The following theorem gives a fractional representation formula in a John domain.
This result is essentially contained in the proof of  \cite[Theorem 4.10]{H-SV}. Therein the constants
need to be tracked more carefully, but this can be done in a straightforward way.

\begin{thm}\label{t.representation}
Let  $0<\tau,\delta<1$ and $M>8/\tau$. Suppose that $G\subset\R^n$ is a $c$-John domain and $u\in L^1_{\textup{loc}}(G)$.
Let $x_0\in G$ be the John center of $G$ and write $B_0=B(x_0,\mathrm{dist}(x_0,\partial G)/(Mc))$. Then there exists a constant $C=C(M,n,c,\delta)>0$ such that inequality
\[
\vert u(x)-u_{B_0}\vert\le C
\int_{G}\frac{g(y)}{\vert x-y\vert ^{n-\delta}}\,dy=C\,\mathcal{I}_{\delta}(\chi_G g)(x)
\]
holds if  $x\in G$ is a Lebesgue point of $u$ and the function $g$ is defined by
\begin{equation*}
g(y)=\int_{B(y,\tau\dist(y,\partial G))}\frac{\vert u(y)-u(z)\vert}{\vert y-z\vert ^{n+\delta}}\,dz\,,\quad y\in G.
\end{equation*}
\end{thm}

The following auxiliary result is from \cite[Lemma 5]{H}.

\begin{lem}\label{MeasureLem}
Let $\gamma$ be a positive measure on a set $X$ with $\gamma(X)<\infty$.
If $\omega\ge 0$ is a measurable function on $X$ such that $\gamma(\{x\in X:\,\omega(x)=0\})\geq \gamma(X)/2$,
then inequality
\[\gamma (\{x\in X:\,\omega(x)>t\})\le 2\inf_{a\in\R}\gamma (\{x\in X:\,\lvert \omega(x)-a\rvert>t/2\})
\]
holds  for every $t>0$.
\end{lem}

\section{Counterexamples}\label{s.counter}

We give an illustrative counterexample  which shows
that the improved  Sobolev--Poincar\'e inequalities are not valid
on bounded Ahlfors $n$-regular domains, in general.
Furthermore, we provide a counterexample showing that, for general John domains,
the seminorms appearing on right hand sides of \eqref{fractionalqp_nonimproved} and \eqref{fractionalqp_I}
are not comparable.

\begin{thm}\label{t.counter}
Let $0<\delta,\tau<1$ and $1< p,q<\infty$ be such that  $1/p-1/q=\delta/n$.
Then there exists a bounded  domain $D$ in $\R^n$ with the following properties.
\begin{itemize}
\item[(A)] $D$ is an Ahlfors $n$-regular domain; in particular,
there exists a constant $C_1>0$ such that inequality
\begin{equation}\label{fractionalqp_usual}
\int_D\vert u(x)-u_D\vert ^q\,dx
\le
C_1
\biggl(\int_D\int_{D}\frac{\vert u(x)-u(y)\vert ^p}{\vert x-y\vert ^{n+\delta p}}\,dy\,dx
\biggr)^{q/p}
\end{equation}
holds for every $u\in L^p(D)$.
\item[(B)] There is no  $C_2>0$ such that the improved fractional $(1,p)$-Poincar\'e inequality
\begin{equation}\label{fractionalqp_H}
\int_D\vert u(x)-u_D\vert\,dx
\le
C_2
\biggl(\int_D\int_{B(x,\tau \dist(x,\partial D))}\frac{\vert u(x)-u(y)\vert ^p}{\vert x-y\vert ^{n+\delta p}}\,dy\,dx
\biggr)^{1/p}
\end{equation}
holds for every $u\in L^\infty(D)$. In particular, the improved fractional $(q,p)$-Poincar\'e inequality
does not hold on $D$.
\end{itemize}
\end{thm}

The proof of Theorem \ref{t.counter} relies on \cite[Theorem 6.9]{H-SV} which we formulate below.

\begin{figure}
\includegraphics[viewport=32 5 112 85,width=6.6cm,clip]{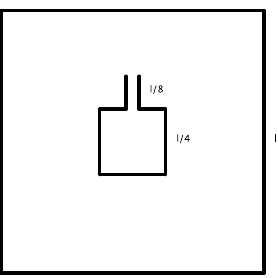}

\textsc{Figure~1.} {\em An $s$-apartment:}\, a room and an $s$-passage in a unit cube.
\end{figure}

\begin{thm}\label{1p_counter}
Let $s>1$, $p\in (1,\infty)$, $\lambda\in [n-1,n)$, and $\delta,\tau\in (0,1)$ be
such that
\[
s<\frac{n+1-\lambda}{1-\delta},\qquad p\le \frac{s(n-1)-\lambda+1}{n-s(1-\delta)-\lambda+1}\,.
\]
Then there exists a bounded domain $G_s\subset \R^n$ satisfying
Êthe following properties: the upper Minkowski dimension of $\partial G_s$ equals $\lambda$ and
 the
fractional $(1,p)$-Poincar\'e inequality \eqref{fractionalqp_H}  does not hold in $D=G_s$
for all functions in $L^\infty(G_s)$. Moreover,
there exists a constant $c\ge 1$ and a point $x_0\in G_s$ such that every $x\in G_s$ can be joined to $x_0$ by a rectifiable
curve $\gamma:[0,\ell]\to G_s$ such that $\dist(\gamma(t),\partial G_s)\ge t^s/c$
for every $t\in [0,\ell]$.
\end{thm}

In the proof of Theorem \ref{1p_counter}
one modifies the usual rooms and $s$-passages construction by placing a room and a passage
of width  $2\ell(Q)^s/8^s$
to each Whitney cube $Q$ of an appropriate John domain $G$, we refer  to Figure~1 from \cite{HHSV}.

\begin{rem}\label{fat}
The domain $G_s$ given by Theorem \ref{1p_counter} is a bounded Ahlfors
$n$-regular domain. Indeed,  the construction begins with a fixed John domain $G$ which is, by
the John condition,
a bounded Ahlfors $n$-regular domain. The domain $G_s$ is then obtained by removing
a set of measure zero from $G$.
We also remark that the usual rooms and $s$-passages construction,
as described in \cite[\S3]{MR2356054}, does not yield an Ahlfors $n$-regular domain if $s>1$.
\end{rem}

\begin{proof}[Proof of Theorem \ref{t.counter}]
Let us fix $\lambda=n-1$ and choose $1<s<2/(1-\delta)$ such that
\[
p\le  \frac{1}{n-s(1-\delta)-\lambda + 1} \le \frac{s(n-1)-\lambda+1}{n-s(1-\delta)-\lambda+1}\,.
\]
Theorem \ref{1p_counter} implies that there exists a bounded domain $D:=G_s$ such that the fractional $(1,p)$-Poincar\'e inequality \eqref{fractionalqp_H} does not hold for all functions in $L^\infty(D)$. Since $q> 1$, the claim (B) follows by H\"older's inequality.

Let us now prove claim (A). By Remark \ref{fat},
the bounded domain $G_s$ is Ahlfors $n$-regular, and
inequality \eqref{fractionalqp_usual} is a consequence of this fact. Indeed,
the embedding $W^{\delta,p}(G_s)\subset L^q(G_s)$ is bounded
 by the Ahlfors $n$-regularity, see e.g. \cite[Theorem 1.2]{Z}.
In particular, there exists a constant $C>0$ such that inequality
\begin{equation}\label{poinc_frac}
\int_{G_s}\vert u(x)-u_{G_s}\vert ^{q}\,dx
\le
C
\biggl(\int_{G_s}\int_{G_s}\frac{\vert u(x)-u(y)\vert^p}{\vert x-y\vert ^{n+\delta p}}\,dy\,dx
+ \| u - u_{G_s}\|_{L^p(G_s)}^p
\biggr)^{q/p}
\end{equation}
holds for each $u\in L^p(G_s)$.
Inequality \eqref{fractionalqp_usual} follows from \eqref{poinc_frac} and the estimate
\begin{align*}
\lVert  u - u_{G_s}\rVert _{L^p(G_s)}^p\le \frac{\diam(G_s)^{n+\delta p}}{\lvert G_s\rvert}\int_{G_s}\int_{G_s}\frac{\vert u(x)-u(y)\vert^p}{\vert x-y\vert ^{n+\delta p}}\,dy\,dx\,.
\end{align*}
\end{proof}

Next we show that inequality \eqref{e.lip} fails for some John domains.

\begin{prop}\label{p.no-comp}
Let $1 \leq p < \infty$ and $0<\delta<1$ with $p\delta \geq 1$, and let $\tau=1$.
Then there exists a~John domain~$G$ for which inequality \eqref{e.lip} fails.
\end{prop}

\begin{proof}
Let $G=(-1,1)^2 \setminus ((0,1) \times \{0\})$.
Let $u:G\to [0,1]$ be defined by
$u(x)=x_1$ for $x\in (0,1)^2$, and $u=0$ otherwise.

We observe that if $x\in G$ and $y\in B(x,\dist(x,\partial G))$,
then $|u(x)-u(y)| \leq |x-y|$, hence the right hand side of \eqref{e.lip}
is finite.

To deal with the left hand side of \eqref{e.lip}, we denote
$L=(1/2,1) \times (-1/4,0)$, and for $x\in L$ we denote
$E(x)=(x_1-|x_2|, x_1) \times (0,|x_2|)$.
Then
\begin{align*}
\int_G\int_{G}\frac{\vert u(x)-u(y)\vert^p}{\vert x-y\vert ^{n+\delta p}}\,dy\,dx
&\geq
4^{-p} \int_L\int_{E(x)} \vert x-y\vert^{-n-\delta p}\,dy\,dx
\geq
c \int_L |x_2|^{-\delta p} \,dx = \infty\,.
\end{align*}
Thus, inequality  \eqref{e.lip} fails.
\end{proof}

\section{From weak to strong}\label{s.weak_to_strong}

The following theorem shows that an improved fractional Poincar\'e inequality of weak type is equivalent to
the corresponding inequality of strong type
if $q\ge p$.

\begin{thm}\label{t.truncation}
Let $\mu$ be a positive Borel measure on an open set $G\subset \R^n$ so that $\mu(G)<\infty$.
Let $0<\delta <1$, $0<\tau\le \infty$, and
$0<p\le q<\infty$.
Then the following conditions are equivalent 
 (with the understanding that $B(y,\tau \dist(y,\partial G)):=\R^n$ whenever
$y\in G$ and $\tau=\infty$):
\begin{itemize}
\item[(A)]
There is a constant $C_1>0$ such that inequality
\begin{align*}
&\quad\inf_{a\in\R}\sup_{t>0} \mu ( \{x\in G:\,|u(x)-a|>t\} ) t^{q}
 \\&\qquad\qquad\qquad  \le C_1
\biggl(\int_{G}\int_{G\cap B(y,\tau \dist(y,\partial G))}\frac{\vert u(y)-u(z)\vert^p}{\vert y-z\vert ^{n+\delta p}}
\,d\mu(z)\,d\mu(y)\biggr)^{q/p}
\end{align*}
holds, for  every $u\in L^\infty(G;\mu)$.
\item[(B)]
There is a constant $C_2>0$ such that inequality
\begin{align*}
\quad \inf_{a\in\R} \int_G\vert u(x)-a\vert ^{q}\,d\mu(x)
\le
C_2
\bigg(\int_G\int_{G\cap B(y,\tau \dist(y,\partial G))}\frac{\vert u(y)-u(z)\vert^p}{\vert y-z\vert ^{n+\delta p}}\,d\mu(z)
\,d\mu(y)\bigg)^{q/p}
\end{align*}
holds, for every $u\in L^1(G;\mu)$.
\end{itemize}
In the implication from {\rm (A)} to {\rm (B)} the constant $C_2$ is of the form $C(p,q)C_1$. In the converse implication
$C_1=C_2$.
\end{thm}

\begin{rem}
Theorem \ref{t.truncation} extends \cite[Theorem 4]{H} to the fractional setting.
The proof  is a combination of an argument in \cite[Theorem 4]{H}
 and a fractional Maz'ya truncation method from
the proof of \cite[Proposition 5]{Dyda3}.
\end{rem}

\begin{proof}[Proof of Theorem \ref{t.truncation}]
The implication from (B) to  (A) is immediate. Let us assume that condition  (A) holds for
all bounded $\mu$-measurable functions.
Fix $u\in L^1(G;\mu)$ and let $b\in\R$ be such that
\[
\mu (\{x\in G:\,u(x)\geq b\}) \geq\frac{\mu(G)}{2}\quad\text{and}\quad\mu (\{x\in G:\,u(x)\leq b\})\geq\frac{\mu(G)}{2}\,.
\]
We write $v_{+}=\max\{u-b,0\}$ and  $v_{-}=-\min\{u-b,0\}$. In the sequel $v$ denotes either $v_{+}$ or $v_{-}$;
all the statements are valid in both cases. Moreover, without loss of generality,
we may assume that $v\ge 0$ is defined and finite everywhere in $G$.

For  $0<t_1<t_2<\infty$ and every $x\in G$, we define
\[
v_{t_1}^{t_2}(x)=
\begin{cases}
t_2-t_1, \qquad &\text{if } v(x)\geq t_2\,,\\
v(x)-t_1, &\text{if } t_1<v(x)<t_2\,,\\
0, &\text{if } v(x)\leq t_1\,.
\end{cases}
\]
Observe that, if $0<t_1<t_2<\infty$, then
\[
\mu (\{x\in G:\,v_{t_1}^{t_2}(x)=0\})\geq \mu(G)/2\,.\]
For $y\in G$â we write $B_{y,\tau}=B(y,\tau \dist(y,\partial G))$.
By Lemma \ref{MeasureLem} and condition  (A), applied to the  function $v_{t_1}^{t_2}\in L^\infty(G;\mu)$,
\begin{equation}\label{e.mas}
\begin{split}
\sup_{t>0} \mu (\{x\in G\,:\,v_{t_1}^{t_2}(x)>t\})
t^{q}&\le 2^{1+q}\inf_{a\in\R}\sup_{t>0}\mu (\{x\in G:\,\lvert v_{t_1}^{t_2}(x)-a\rvert >t\}) t^q \\
&\le 2^{1+q}C_1 \biggl(\int_{G}\int_{G\cap B_{y,\tau}}\frac{\lvert v_{t_1}^{t_2}(y)-v_{t_1}^{t_2}(z)\rvert^p}{\lvert y-z\rvert ^{n+\delta p}}\,d\mu(z)\,d\mu(y)
\biggr)^{q/p}\,.
\end{split}
\end{equation}
We write
$E_k = \{x\in G \,:\,  v(x) > 2^{k} \}$ and $A_k = E_{k-1} \setminus E_{k}$, where $k\in \Z$.
Since $v\ge 0$ is finite everywhere,  we can write
\begin{equation}\label{e.decomp}
G=  \{x\in G\,:\, 0\le v(x)<\infty\} = \bigcup_{i\in \Z}A_i \cup \underbrace{\{x\in G \,:\, v(x)=0\}}_{=:A_{-\infty}}\,.
\end{equation}
Hence, by inequality \eqref{e.mas} and the fact that $\sum_{k\in \Z} |a_k|^{q/p} \leq (\sum_{k\in \Z} |a_k|)^{q/p}$, we obtain that
\begin{align*}
\int_G \lvert v(x)\rvert^q \,d\mu(x)
&\le  \sum_{k\in \Z} 2^{(k+1)q} \mu(A_{k+1})\\
&\le
\sum_{k\in\Z}  2^{(k+1)q} \mu(\{x\in G:v_{2^{k-1}}^{2^{k}}(x)\geq 2^{k-1}\})\\&\le
2^{1+4q}C_1\biggl(\sum_{k\in \Z} \int_{G}\int_{G\cap B_{y,\tau}}\frac{\vert v_{2^{k-1}}^{2^{k}}(y)-v_{2^{k-1}}^{2^{k}}(z)\vert^p}{\vert y-z\vert ^{n+\delta p}}\,d\mu(z)\,d\mu(y)\biggr)^{q/p}\,.
\end{align*}  
By \eqref{e.decomp} we can estimate
\begin{align}\label{e.remaining}
&\sum_{k\in \Z}  \int_{G}\int_{G\cap B_{y,\tau}}\frac{\vert v_{2^{k-1}}^{2^{k}}(y)-v_{2^{k-1}}^{2^{k}}(z)\vert^p}{\vert y-z\vert ^{n+\delta p}}\,d\mu(z)\,d\mu(y) \notag \\
& \le \bigg\{ \sum_{k\in \Z}  \sum_{ -\infty \le i\le k} \sum_{j\ge k} \int_{A_i}\int_{A_j\cap B_{y,\tau}}
\notag 
\\
&\quad \qquad 
+  \sum_{k\in \Z}   \sum_{i\ge k} \sum_{-\infty\le j\le k} \int_{A_i}\int_{A_j\cap B_{y,\tau}} 
\bigg\}
\frac{\vert v_{2^{k-1}}^{2^{k}}(y)-v_{2^{k-1}}^{2^{k}}(z)\vert^p}{\vert y-z\vert ^{n+\delta p}}\,d\mu(z)\,d\mu(y)\,.
\end{align}
Let $y\in A_i$ and $z\in A_j$, where $j-1> i \ge -\infty$.
Then $ \lvert v(y)-v(z)\rvert \geq \lvert v(z)\rvert  - \lvert v(y)\rvert  \geq 2^{j-2}$. Hence,
\begin{equation}\label{summandEstimate}
\lvert v_{2^{k-1}}^{2^{k}}(y)-v_{2^{k-1}}^{2^{k}}(z)\rvert \le 2^k\le 4\cdot 2^{k-j}\lvert v(y)-v(z)\rvert\,.
\end{equation}
Since the estimate
\[
\lvert v_{2^{k-1}}^{2^{k}}(y)-v_{2^{k-1}}^{2^{k}}(z)\rvert  \leq \lvert v(y)-v(z)\rvert
\]
holds for every $k\in\Z$,
inequality \eqref{summandEstimate} is valid whenever $-\infty\le i\le k\le j$
and $(y,z)\in A_i\times A_j$.
By
 inequality \eqref{summandEstimate}:
\begin{equation}\label{ThreeSums}
\begin{split}
\sum_{k\in \Z}\sum_{-\infty\le i\le k} &\sum_{j\ge k} \int_{A_i}\int_{A_j\cap B_{y,\tau}}\frac{\vert v_{2^{k-1}}^{2^{k}}(y)-v_{2^{k-1}}^{2^{k}}(z)\vert^p}{\vert y-z\vert ^{n+\delta p}}\,d\mu(z)\,d\mu(y)\\
 &\le  4^p \sum_{k\in \Z} \sum_{-\infty\le i\leq k} \sum_{j\geq k} 2^{p(k-j)} \int_{A_i} \int_{A_j\cap B_{y,\tau}} \frac{\lvert v(y)-v(z)\rvert^p}{\lvert y-z\rvert^{n+\delta p}}\,d\mu(z)\,d\mu(y).
\end{split}
\end{equation}
Since $\sum_{k=i}^j 2^{p(k-j)}  \le (1-2^{-p})^{-1}$, changing the order of the summation yields
that the right hand side of inequality \eqref{ThreeSums} is bounded by
\[
 \frac{4^p}{1-2^{-p}} 
\int_G \int_{G\cap B_{y,\tau}} \frac{|v(y)-v(z)|^p}{|y-z|^{n+\delta p}}\,d\mu(z)\,d\mu(y)\,.
\]
The estimation of the 
second term
 in \eqref{e.remaining} is  also performed as above.
To conclude that  (B) holds with $C_2=C(q,p)C_1$ it remains to recall that $\lvert u-b\rvert=v_{+}+v_{-}$ and $q>0$. Observe also that
$\lvert v_\pm(y)-v_\pm(z)\rvert \le \lvert u(y)-u(z)\rvert$ for all $y,z\in G$.
\end{proof}

\begin{rem}\label{r.suffices}
If $q\ge 1$ in  Theorem \ref{t.truncation}, then we may replace the
infimum on the
left hand side of the inequality appearing in condition {\rm (B)} by
$\int_G\vert u(x)-u_{G;\mu}\vert ^{q}\,d\mu(x)$.
Indeed, by H\"older's inequality,
\[
\int_G\vert u(x)-u_{G;\mu}\vert ^{q}\,d\mu(x)
\le 2^q\, \inf_{a\in\R}\int_G\vert u(x)-a\vert ^{q}\,d\mu(x)\,.
\]
Here we have written $u_{G;\mu}=\frac{1}{\mu(G)}\int_G u(y)\,d\mu(y)$.
\end{rem}

\color{black}

\section{Improved fractional Sobolev--Poincar\'e inequality}\label{s.fractional}

R. Hurri-Syrj\"anen and the third author prove in \cite[Theorem 4.10]{H-SV} an improved
fractional Sobolev--Poincar\'e inequality on a given $c$-John domain $G$. Namely, let
us fix $0<\delta,\tau<1$ and $1< p<n/\delta$.
Then there exists a constant $C=C(n,\delta,c,\tau,p)$ such that inequality
\begin{equation}\label{fractionalqp}
\int_G\vert u(x)-u_G\vert ^{np/(n-\delta p)}\,dx
\le
C
\biggl(\int_G\int_{B(x,\tau \dist(x,\partial G))}\frac{\vert u(x)-u(y)\vert ^p}{\vert x-y\vert ^{n+\delta p}}\,dy\,dx
\biggr)^{n/(n-\delta p)}
\end{equation}
holds for every $u\in L^1(G)$.

We prove inequality \eqref{fractionalqp} when $p=1$.

\begin{thm}\label{t.application}
Suppose that $G$ is a $c$-John domain in $\R^n$ and
let $\tau,\delta\in (0,1)$ be given.
Then there exists a constant $C=C(n,\delta,c,\tau)>0$ such that inequality
\begin{equation*}
\int_G\vert u(x)-u_G\vert ^{n/(n-\delta)}\,dx
\le
C
\biggl(\int_G\int_{B(x,\tau\dist(x,\partial G))}\frac{\vert u(x)-u(y)\vert}{\vert x-y\vert ^{n+\delta}}\,dy\,dx
\biggr)^{n/(n-\delta)}
\end{equation*}
holds for every $u\in L^1(G)$.
\end{thm}

\begin{proof}
By Theorem \ref{t.truncation} and Remark \ref{r.suffices}, it suffices
to prove that there exists a constant $C=C(n,\delta,c,\tau)>0$ such that inequality
\begin{equation}\label{WeakSobPoincare}
\begin{split}
&\inf_{a\in\R}\sup_{t>0}\big|\{x\in G:\,\lvert u(x)-a\rvert>t\}\big|t^{n/(n-\delta)}
\\&\qquad\qquad\le C
\biggl(\int_{G}\int_{B(y,\tau \dist(y,\partial G))}\frac{\vert u(y)-u(z)\vert}{\vert y-z\vert ^{n+\delta}}\,dz\,dy\biggr)^{n/(n-\delta)}
\end{split}
\end{equation}
holds for every $u\in L^\infty(G)$.
Let us denote by $x_0\in G$ the John center of $G$, and let
\[
B_0:=B(x_0,\mathrm{dist}(x_0,\partial G)/(Mc))\,,
\]
where $M=9/\tau$.
We also write
\begin{equation*}
g(y)=\int_{B(y,\tau\dist(y,\partial G))}\frac{\vert u(y)-u(z)\vert}{\vert y-z\vert ^{n+\delta}}\,dz
\end{equation*}
for every $y\in G$.
By Theorem \ref{t.representation},  for each Lebesgue point $x\in G$ of $u$,
\begin{equation}\label{OscillationEstimate}
\vert u(x)-u_{B_0}\vert\le C(n,c,\delta,\tau)
\int_{G}\frac{g(y)}{\vert x-y\vert ^{n-\delta}}\,dy=C(n,c,\delta,\tau)\,\mathcal{I}_{\delta}(\chi_G g)(x)\,.
\end{equation}
By inequality \eqref{OscillationEstimate} and Theorem \ref{AdamsHedberg},
there exists a constant $C=C(n,c,\delta,\tau)$ such that
\begin{align*}
&\big|\{x\in G:\,|u(x)-u_{B_0}|>t\}\big|t^{n/(n-\delta)}\\&\qquad\qquad \le C
\biggl(\int_{G}\int_{B(y,\tau \dist(y,\partial G))}\frac{\vert u(y)-u(z)\vert}{\vert y-z\vert ^{n+\delta}}\,dz\,dy\biggr)^{n/(n-\delta)}
\end{align*}
for every $t>0$. Inequality \eqref{WeakSobPoincare} follows.\end{proof}

\begin{rem}
Inequality \eqref{fractionalqp} makes sense only if the domain $G$ has a finite measure.
If we replace the left hand side of inequality \eqref{fractionalqp} by
\[
\int_G \lvert u(x)\rvert^{np/(n-\delta p)}\,dx\,,
\]
then the resulting inequality is valid
on so-called  unbounded John domains $G$
that are of infinite measure, we refer  to \cite[\S5]{H-SV2}.
\end{rem}

\section{Necessary conditions for the improved inequality}\label{s.necessary}

In this section, we obtain necessary conditions for the improved Poincar\'e inequalities.
Theorem \ref{t.necessary} gives a counterpart for the result of Buckley and Koskela on the classical
Sobolev--Poincar\'e inequality \eqref{sobolev_poincare}, see \cite[Theorem 1.1]{MR1359964}.

\begin{thm}\label{t.necessary}
Assume that $G$ is a domain of finite measure in $\R^n$  which satisfies the separation property.
Let $\delta\in (0,1)$ and $1\le p<n/\delta$ be given.
If
there exists a constant $C_1>0$  such that the improved fractional Sobolev--Poincar\'e inequality
\begin{equation}\label{e.improved_sp}
\int_G \lvert u(x)-u_G\rvert^{np/(n-\delta p)}\,dx \le C_1\bigg(\int_G \int_{B(x,\dist(x,\partial G))}
\frac{\lvert u(x)-u(y)\rvert^p}{\lvert x-y\rvert^{n+\delta p}}\,dy\,dx\bigg)^{n/(n-\delta p)}
\end{equation}
holds for every $u\in L^\infty(G)$, then $G$ is a John domain.
\end{thm}

To prove Theorem \ref{t.necessary} it suffices to
prove Proposition \ref{p.zero}, and then follow the geometric arguments given in \cite[pp. 6--7]{MR1359964}.
Observe that $(1/p-1/q)/\delta =1/n$ and $(n-\delta p)q/(np)= 1$  if $q=np/(n-\delta p)$.

\begin{prop}\label{p.zero}
Suppose that $G\subset \R^n$ is a domain of finite measure. Let
$\delta\in (0,1)$ and
$1\le p< q<\infty$ be given. Assume
that there exists a constant $C_1>0$ such that inequality
\begin{equation}\label{e.poinc}
\bigg(\int_G \lvert u(x)-u_G\rvert^{q}\,dx \bigg)^{1/q} \le C_1\bigg(\int_G \int_{B(x,\dist(x,\partial G))}
\frac{\lvert u(x)-u(y)\rvert^p}{\lvert x-y\rvert^{n+\delta p}}\,dy\,dx\bigg)^{1/p}
\end{equation}
holds for every $u\in L^\infty(G)$.
Fix a ball $B_0\subset G$, and let $d>0$ and $w\in G$.
 Then there exists a constant $C>0$ such that
\begin{equation}\label{e.T-est}
\diam(T)\le C(d+|T|^{(\frac{1}{p}-\frac{1}{q})\frac{1}{\delta}}) \qquad \text{and}\qquad
 \lvert T\rvert^{1/n}\le C(d+ d^{(n-\delta p)q/(np)})
\end{equation}
if $T$ is the union of all components of $G\setminus B(\omega,d)$ that do not intersect the ball $B_0$.
The constant $C$  depends on $C_1$, $\lvert B_0\rvert$, $\lvert G\rvert$, $n$, $\delta$, $q$, and $p$
only.
\end{prop}

Notice that inequalities in~\eqref{e.T-est} extend \cite[Theorem 2.1]{MR1359964} to the fractional case.

\begin{proof}[Proof of the first inequality in \eqref{e.T-est}]
Without loss of generality, we may assume that $T\not=\emptyset$.
Let $T(r)=T\setminus B(\omega,r)$, we will later prove inequality
\begin{equation}\label{e.postponed}
\lvert T(r)\rvert^{p/q} \leq \frac{c\lvert T(\rho)\rvert}{(r-\rho)^{\delta p}}\,,
\end{equation}
provided $d\leq \rho < r$.
Assuming that this inequality holds, one proceeds as follows.
Define $r_0=d$ and for $j\geq 1$ pick $r_j>r_{j-1}$ such that
\[
\lvert A(r_{j-1},r_j)\rvert=\lvert T\cap B(w,r_j)\setminus  B(w,r_{j-1})\rvert =2^{-j}\lvert T\rvert\,.
\]
Then
$\lvert T(r_j)\rvert =\lvert T\setminus  B(w,r_j)\rvert=2^{-j}|T|$.
Hence, by inequality \eqref{e.postponed}
\begin{align*}
\diam(T)&\le 2d+\sum\limits_{j=1}^\infty 2\lvert r_j-r_{j-1}\rvert \\&\le 2d+c\sum\limits_{j=1}^\infty(\lvert T(r_{j-1})\rvert \lvert T(r_j)\rvert^{-p/q})^{\frac{1}{\delta p}}
\\&=2d+c\sum\limits_{j=1}^\infty(2^{-j+1}\lvert T\rvert 2^{jp/q}\lvert T\rvert^{-p/q})^{\frac{1}{\delta p}}
\\&=2d+c\lvert T\rvert^{(\frac{1}{p}-\frac{1}{q})\frac{1}{\delta}}\sum\limits_{j=1}^\infty 2^{-j(\frac{1}{p}-\frac{1}{q})\frac{1}{\delta}}\le 2d+c\lvert T\rvert^{(\frac{1}{p}-\frac{1}{q})\frac{1}{\delta}}
\end{align*}
and this concludes the main line of the argument.

It remains to prove inequality \eqref{e.postponed}.
We assume that $T(r)\neq\emptyset$
 and define a bounded function $u$ on $G$ as follows
\[
u(x)=
\begin{cases}
1\,,  &\quad x\in T(r)\,,\\
\frac{\dist(x,B(\omega,\rho))}{r-\rho}\,,&\quad x\in A(\rho,r)=T(\rho)\setminus T(r)\,,\\
0\,, &\quad x\in G\setminus T(\rho)\,.
\end{cases}
\]
For $x\in G$, let us denote $B_{x,1} = B(x, \dist(x,\partial G))$.
By the fact that $u=0$ on $B_0$ and inequality \eqref{e.poinc} we obtain
\begin{equation}\label{NecessityFirstIneq}
\begin{split}
\lvert T(r)\rvert^{p/q}&\le \biggl(\int_G\vert u(x)\vert ^q\,dx\biggr)^{p/q}\\&\le c\biggl(\int_G\lvert u(x)-u_G\rvert ^q\,dx\biggr)^{p/q}\le c\int_G\int_{B_{x,1}}\frac{\lvert u(x)-u(y)\rvert ^p}{\vert x-y\vert ^
{n+\delta p}}\,dy\,dx\,.
\end{split}
\end{equation}
For all measurable $E,F\subset G$, denote
\[
I(E,F) = \int_E \int_{B_{x,1}\cap F} \frac{\lvert u(x)-u(y)\rvert^p}{\lvert x-y\rvert^{n+\delta p}}\,dy\,dx\,.
\]
Since $u=0$ on $G\setminus T(\rho)$ and $u=1$ on $T(r)$, we can write the right hand side of \eqref{NecessityFirstIneq} as
\begin{equation}\label{FracNormSummands}
\begin{split}
I(G,G)=&I(T(r), A(\rho,r)) +I(T(r),G\setminus T(\rho)) \\&
+ I(A(\rho,r),T(r)) + I(A(\rho,r),A(\rho,r)) +
 I(A(\rho,r),G\setminus T(\rho)) \\&
+ I(G\setminus T(\rho),T(r)) + I(G\setminus T(\rho),A(\rho,r))\,.
\end{split}
\end{equation}

For the first and the third term of \eqref{FracNormSummands} we use the following estimate
\[
I(T(r), A(\rho,r))+I(A(\rho,r), T(r))\le 2\int_{A(\rho,r)}\int_{T(r)}\frac{\vert u(x)-u(y)\vert^p}{\vert x-y\vert^{n+\delta p}}\,dy\,dx.
\]
We observe that, for every $x\in A(\rho,r)$,
\[\lvert \dist(x,B(\omega,\rho))-(r-\rho)\rvert\le\min\{ \dist(x,T(r)), r-\rho \}=m(x)\,.\]
By the definition of function $u$,
\begin{align*}
&\int_{A(\rho,r)}\int_{T(r)}\frac{\vert \dist(x,B(\omega,\rho))-(r-\rho)\vert ^p}{(r-\rho)^{p}\vert x-y\vert^{n+\delta p}}\,dy\,dx\le
\int_{A(\rho,r)}\int_{T(r)}\frac{m^p(x)}{(r-\rho)^{p}\vert x-y\vert^{n+\delta p}}\,dy\,dx\\
&\le \int_{A(\rho,r)}\int_{\R^n\setminus B(x,m(x))}\frac{m^p(x)}{(r-\rho)^{p}\vert x-y\vert^{n+\delta p}}\,dy\,dx=  c\int_{A(\rho,r)}\frac{(m(x))^{p-\delta p}}{(r-\rho)^p}\,dx \le \frac{c\lvert A(\rho,r)\rvert}{(r-\rho)^{\delta p}}.
\end{align*}

We estimate the second term $I(T(r),G\setminus T(\rho))$. Let us show that, for every $x\in T(r)$,
\begin{equation}\label{Inclusion}
B_{x,1} \cap (G\setminus T(\rho))\subset \R^n\setminus B(x,r-\rho)\,.
\end{equation}
If $y\in G\setminus T(\rho)$, then the point $y$  belongs to the ball $B(\omega,\rho)$ or to a component of
$G\setminus B(\omega,d)$ that intersects the ball $B_0$. At the same time, if $y\in B_{x,1}$, then $B(x,\lvert x-y\rvert)\subset G$ which
means that the situation when $x$ and $y$ are in different components of $G\setminus B(\omega,d)$ is not possible. Hence,
$y\in B(\omega,\rho)$, and indeed $
\lvert x-y\rvert\ge\lvert x-w\rvert-\lvert w-y\rvert \ge r-\rho$.

By \eqref{Inclusion}, for each $x\in T(r)$, we have
\[
\int_{B_{x,1} \cap (G\setminus T(\rho))}\frac{1}{\vert x-y\vert ^{n+\delta p}}\,dy\le \int_{\R^n\setminus B(x,r-\rho)}\frac{1}{\vert x-y\vert ^{n+\delta p}}\,dy=
c(r-\rho)^{-\delta p}\,,
\]
and hence
\[
I(T(r),G\setminus T(\rho))\le c\frac{|T(r)|}{(r-\rho)^{\delta p}} \le c\frac{|T(\rho)|}{(r-\rho)^{\delta p}}\,.
\]

Next we consider $I(A(\rho,r),A(\rho,r))$. Notice that, for every $x\in A(\rho,r)$,
\begin{align*}
&\int_{B_{x,1} \cap A(\rho,r)}\frac{\vert \dist(x,B(\omega,\rho))-\dist(y,B(\omega,\rho))\vert ^p}{(r-\rho)^{p}\vert x-y\vert ^{n+\delta p}}\,dy
\\
&\le (r-\rho)^{-p}\int_{A(\rho,r)\cap B(x,r-\rho)}\frac{1}{\vert x-y\vert ^{n+\delta p-p}}\,dy+
\int_{A(\rho,r)\setminus B(x,r-\rho)}\frac{ 1}{\vert x-y\vert ^{n+\delta p}}\,dy\\
&\le  \frac{c(r-\rho)^{p-\delta p}}{(r-\rho)^p}+ \frac{c}{(r-\rho)^{\delta p}}\,.
\end{align*}
Hence, we obtain that
\[
I(A(\rho,r),A(\rho,r))\le c\frac{|A(\rho,r)|}{(r-\rho)^{\delta p}}\,.
\]

Then we focus on $I(A(\rho,r),G\setminus T(\rho))$.
Let us first observe that, for every $x\in A(\rho,r)$,
\[B_{x,1} \cap (G\setminus T(\rho))\subset \R^n\setminus B(x,\dist(x,B(\omega,\rho)))\,.\]
To verify this, we fix $y\in B_{x,1} \cap (G\setminus T(\rho))$. By repeating the argument used in the proof of
inclusion \eqref{Inclusion} we obtain that $y\in B(\omega,\rho)$ and
$\lvert y-x\rvert \ge \dist(x,B(\omega,\rho))$.
Thus, for every $x\in A(\rho,r)$,
\[
\int_{B_{x,1} \cap (G\setminus T(\rho))}\frac{1}{\vert x-y\vert ^{n+\delta p}}\,dy\le \int_{\R^n\setminus B(x,\dist(x,B(\omega,\rho)))}
\frac{1}{\vert x-y\vert ^{n+\delta p}}\,dx= 
c(\dist(x,B(\omega,\rho)))^{-\delta p}\,.
\]
Therefore, we have
\[
I(A(\rho,r),G\setminus T(\rho))\le c\int_{A(\rho,r)}\frac{(\dist(x,B(\omega,\rho)))^{p-\delta p}}{(r-\rho)^p}\,dx\le \frac{c\lvert A(\rho,r)\rvert}{(r-\rho)^{\delta p}}\,.
\]

In order to estimate the remaining terms $I(G\setminus T(\rho),T(r))$ and $I(G\setminus T(\rho),A(\rho,r))$ we observe that,
if $x\in G\setminus T(\rho)$ and $B_{x,1} \cap T(\rho)\neq\emptyset$, then $x\in B(w,\rho)$. This follows from the fact that, if $y\in B_{x,1}\cap T(\rho)$ then $B(x,|x-y|)\subset G$ and,
hence, $x$ and $y$ can not belong to different components of $G\setminus B(\omega,\rho)$.

Using the observation above and adapting the estimates for the term $I(T(r), G\setminus T(\rho))$, we obtain
\begin{align*}
I(G\setminus T(\rho),T(r))&=I(B(\omega,\rho) \cap G,T(r))\\
&\le \int_{T(r)}\int_{B(\omega,\rho)}\frac{\vert u(x)-u(y)\vert^p}{\vert x-y\vert^{n+\delta p}}\,dy\,dx\le c\frac{|T(\rho)|}{(r-\rho)^{\delta p}}\,.
\end{align*}
Following the same argument and adapting the estimates for $I(A(\rho,r), G\setminus T(\rho))$ we obtain
 that $I(G\setminus T(\rho),A(\rho,r))\le c|A(\rho,r)|(r-\rho)^{-\delta p}$.
\end{proof}

We proceed to the second part of Proposition~\ref{p.zero}.

\begin{proof}[Proof of the second inequality in \eqref{e.T-est}]
We first observe that
$\lvert T\rvert \le Cd^n + \lvert T(2d)\rvert$. Hence, it remains to show that
\begin{equation}\label{T2dMeasure}
\lvert T(2d)\rvert\le Cd^{(n-\delta p)q/p}\,.
\end{equation}
In order to do this, we use a slightly modified proof of the first inequality. More precisely, by inequality \eqref{NecessityFirstIneq}, for $d\leq \rho < r$, we have
\[
\lvert T(r)\rvert^{p/q}\le I(G,G)\,,
\]
where $I(G,G)$ can be written as in \eqref{FracNormSummands}.
From the reasoning above it is seen that all the terms in \eqref{FracNormSummands} except $I(T(r), G\setminus T(\rho))$ and $I(G\setminus T(\rho),T(r))$ are bounded from above by
$
c|A(\rho,r)|(r-\rho)^{-\delta p}
$. Furthermore, for the remaining terms, we have
\begin{align*}
&I(T(r),G\setminus T(\rho))+I(G\setminus T(\rho),T(r))=I(T(r), B(\omega,\rho)\cap G)+I(B(\omega,\rho)\cap G,T(r))\\
&\le 2\int_{B(\omega,\rho)} \int_{T(r)}\frac{1}{\vert x-y\vert^{n+\delta p}}\,dy\,dx
\le 2\int_{B(\omega,\rho)}\int_{\R^n\setminus B(x,r-\rho)}\frac{1}{\vert x-y\vert^{n+\delta p}}\,dy\,dx
\le c\frac{|B(\omega,\rho)|}{(r-\rho)^{\delta p}}\,.
\end{align*}
Thus,
\[
\lvert T(r)\rvert^{p/q}\le \frac{c}{(r-\rho)^{\delta p}}\big(|A(\rho,r)|+ |B(\omega,\rho)|\big)\,.
\]
Next we set $\rho=d$ and $r=2d$ in the inequality above, and using the trivial estimates for the measures of a ball and of an annulus, we obtain \eqref{T2dMeasure}.
\end{proof}

\bibliographystyle{abbrv}
\def\cprime{$'$} \def\cprime{$'$} \def\cprime{$'$}

\end{document}